\documentclass[11pt,reqno]{amsart}
\usepackage[margin=1.05in]{geometry}
\usepackage{amsmath,amssymb,amsthm,mathtools}
\usepackage{microtype}
\usepackage[hidelinks,hypertexnames=false]{hyperref}

\newcommand{\Irr}{\operatorname{Irr}}
\newcommand{\cod}{\operatorname{cod}}
\newcommand{\Aut}{\operatorname{Aut}}
\newcommand{\Stab}{\operatorname{Stab}}
\newcommand{\Out}{\operatorname{Out}}
\newcommand{\PSU}{\operatorname{PSU}}
\newcommand{\PSp}{\operatorname{PSp}}
\newcommand{\SU}{\operatorname{SU}}

\newcommand{\PSL}{\operatorname{PSL}}

\newcommand{\OmegaG}{\operatorname{\Omega}}

\newtheorem{theorem}{Theorem}[section]
\newtheorem{proposition}[theorem]{Proposition}
\newtheorem{lemma}[theorem]{Lemma}

\theoremstyle{remark}
\newtheorem{remark}[theorem]{Remark}

\title{Element Orders and Character Codegrees of Finite Groups}
\author{YONG YANG}
\address{Department of Mathematics, Texas State University, San Marcos, TX 78666, USA.}
\email{yang@txstate.edu}
\date{}

\begin{document}
\begin{abstract}
We prove Qian's conjecture on element orders and irreducible character codegrees.
\end{abstract}

\maketitle

\section{Introduction}

For $\chi\in\Irr(G)$, write
\[
\cod_G(\chi)=\frac{|G:\ker\chi|}{\chi(1)}.
\]
Qian conjectured that, for every finite group $G$ and every $g\in G$, there exists $\chi\in\Irr(G)$ such that
\begin{equation}
 o(g)\mid \cod_G(\chi).
\tag{1.1}
\end{equation}
He proved the conjecture for solvable groups \cite[Theorem~B]{Qian2021}. Madanha proved it for almost simple groups \cite[Theorem~A]{Madanha2023}; this includes the nonsolvable symmetric and alternating groups. Akhlaghi, Pacifici and Sanus proved it when $F(G)=1$ and showed that every minimal normal subgroup of a minimal counterexample is abelian; see \cite[Theorem~A and Remark~3.1]{APS2024}.

We prove the following theorem, and thereby prove Qian's conjecture.

\begin{theorem}\label{thm:qian-main}
For every finite group $G$ and every $g\in G$, there exists $\chi\in\Irr(G)$ such that
\[
 o(g)\mid\cod_G(\chi).
\]
\end{theorem}

We extend Qian's induction to nonabelian chief factors. For a simple factor the problem becomes a degree--orbit divisibility for a quasisimple group with cyclic center. This is the degree part of Navarro--Tiep's \cite[Condition~2.1]{NavarroTiep2026}; their component-permutation argument \cite[Theorem~2.4]{NavarroTiep2026} can then be used without the additional restriction condition. Double covers of alternating groups are treated with spin characters, the ordinary Lie-type covers with semisimple characters and regular embeddings, and the remaining twenty-four exceptional cyclic covers by their ordinary character tables.

\section{The relative induction}

We begin with Qian's relative result.

\begin{proposition}\label{prop:qian-relative}
Let $Z\lhd G$, let $\lambda\in\Irr(Z)$, and suppose that $G/Z$ is solvable. If $x\in I_G(\lambda)$, then there exists $\chi\in\Irr(\lambda^G)$ such that
\begin{equation}
 o(xZ)\chi(1)\mid |G:Z|\lambda(1).
\tag{2.1}
\end{equation}
\end{proposition}

This is \cite[Proposition~2.2]{Qian2021}. We recall the part of the proof that is relevant here. If $\lambda$ is not $G$-invariant, one passes to $I_G(\lambda)$ and then induces back. After replacing the character triple by an isomorphic one, one may assume that $Z\le Z(G)$ and that $\lambda$ is faithful and linear. Let $L/Z$ be a chief factor. Since $G/Z$ is solvable, $L/Z$ is elementary abelian.

Qian then separates two possibilities. In the fully ramified case there is a unique character $\theta\in\Irr(L\mid\lambda)$ and
\[
 \lambda^L=e\theta,\qquad \theta(1)=e,\qquad |L:Z|=e^2.
\]
The induction hypothesis applied above $L$ gives the required divisibility. In the extension case $\lambda$ has $|L:Z|$ extensions to $L$, all linear. One chooses an extension with a controlled orbit under the $p$-part of $\langle x\rangle$ and then uses the elementary abelian structure of $L/Z$ to show that the only possible loss in element order is absorbed by the orbit size. The solvability assumption enters exactly through the fact that $L/Z$ is abelian; see \cite[Proposition~2.2]{Qian2021}.

\begin{proposition}\label{prop:relative-reduction}
Assume that whenever $Z\le Z(G)$, $\lambda\in\Irr(Z)$ is faithful, $x\in G$, and $L/Z$ is a nonabelian chief factor of $G$, there exists $\theta\in\Irr(L\mid\lambda)$ such that, with
\[
 c=|\theta^{\langle x\rangle}|,
\]
one has
\begin{equation}
 \frac{o(xZ)}{o(x^cL)}\cdot\theta(1)\mid |L:Z|.
\tag{2.2}
\end{equation}
Then Proposition~\ref{prop:qian-relative} holds without the hypothesis that $G/Z$ is solvable.
\end{proposition}

\begin{proof}
We argue by induction on $|G:Z|$. Every later application of the induction hypothesis has a strictly smaller quotient index. Suppose first that $\lambda$ is not $G$-invariant and set $T=I_G(\lambda)$. Since $x\in T<G$, induction in $T$ gives $\xi\in\Irr(T\mid\lambda)$ with
\[
 o(xZ)\xi(1)\mid |T:Z|\lambda(1).
\]
Then $\chi=\xi^G$ is irreducible by Clifford theory and
\[
 \chi(1)=|G:T|\xi(1),
\]
so
\[
 o(xZ)\chi(1)\mid |G:Z|\lambda(1).
\]
Thus we may assume that $\lambda$ is $G$-invariant.

By the character-triple reduction used in \cite[proof of Proposition~2.2]{Qian2021}, equivalently \cite[Theorem~11.28]{Isaacs1976}, we may replace $(G,Z,\lambda)$ by an isomorphic character triple $(G^*,Z^*,\lambda^*)$ with $Z^*\le Z(G^*)$ and $\lambda^*$ faithful and linear. The triple isomorphism identifies $G/Z$ with $G^*/Z^*$; we take $x^*Z^*$ corresponding to $xZ$. Thus $o(x^*Z^*)=o(xZ)$, and inertia groups and relative character degrees are preserved. We therefore suppress the stars and assume that $Z\le Z(G)$ and $\lambda$ is faithful and linear. Let $L/Z$ be a chief factor of $G$.

If $L/Z$ is abelian, it is elementary abelian. The remainder of Qian's proof of \cite[Proposition~2.2]{Qian2021} uses only this fact together with the induction hypothesis for smaller quotient index, so the same argument applies here. We therefore assume that $L/Z$ is nonabelian and choose $\theta$ as in the hypothesis. Put
\[
 T=I_G(\theta),\qquad y=x^c.
\]
By the definition of $c$, we have $y\in T$. The character $\theta$ is $T$-invariant and $|T:L|<|G:Z|$, so induction applied to $(T,L)$ gives $\xi\in\Irr(T\mid\theta)$ satisfying
\begin{equation}
 o(yL)\xi(1)\mid |T:L|\theta(1).
\tag{2.3}
\end{equation}
Now $\chi=\xi^G\in\Irr(G)$ and $\chi(1)=|G:T|\xi(1)$. Multiplying (2.3) by $|G:T|$ gives
\[
 o(yL)\chi(1)\mid |G:L|\theta(1).
\]
Combining this with (2.2), and recalling that $y=x^c$, we obtain
\[
 o(xZ)\chi(1)\mid |G:Z|.
\]
Since $\lambda(1)=1$ in the reduced triple, this is exactly the desired relative divisibility.
\end{proof}

\section{One simple component}

Assume first that
\[
 L/Z\cong S
\]
is nonabelian simple. Set
\begin{equation}
 d=o(xL),\qquad n=o(xZ),\qquad r=\frac nd.
\tag{3.1}
\end{equation}
Then $x^d\in L$, and the image $x^dZ$ has order $r$ in $L/Z$.

Let $\theta\in\Irr(L\mid\lambda)$ and put
\[
 c=|\theta^{\langle x\rangle}|.
\]
Since $x^d\in L$, conjugation by $x^d$ is inner on $L$ and hence fixes every irreducible character of $L$. Therefore
\begin{equation}
 c\mid d.
\tag{3.2}
\end{equation}
It follows that
\[
 o(x^cL)=\frac dc
\]
and hence
\begin{equation}
 \frac{o(xZ)}{o(x^cL)}=rc.
\tag{3.3}
\end{equation}
Thus the local condition (2.2) becomes
\begin{equation}
 rc\,\theta(1)\mid |S|.
\tag{3.4}
\end{equation}

Since $L/Z$ is perfect,
\[
 L=L'Z.
\]
Put
\[
 Q=L',\qquad C=Q\cap Z.
\]
Then $Q/C\cong L/Z\cong S$. Moreover $C=Z(Q)$. Indeed, $C\le Z(Q)$ because $C\le Z$, while $Z(Q)C/C\le Z(Q/C)=1$. Thus
\begin{equation}
 Z(Q)=C,\qquad Q/C\cong S.
\tag{3.5}
\end{equation}
Since $C\le Z$ and $\lambda$ is faithful and linear, $C$ is cyclic. Put
\[
 \nu=\lambda_C.
\]
Then $\nu$ is faithful. For $u\in Q$, let $\iota_u$ denote the inner automorphism of $Q$ given by $v\mapsto u^{-1}vu$. If $\alpha\in\Irr(Q\mid\nu)$, define
\[
 \theta(qz)=\alpha(q)\lambda(z)\qquad(q\in Q,\ z\in Z).
\]
This is well defined because $Q\cap Z=C$ and $\alpha_C=\alpha(1)\nu=\alpha(1)\lambda_C$. It is an irreducible character of $L=QZ$, lies over $\lambda$, and has degree $\alpha(1)$. Conversely every character in $\Irr(L\mid\lambda)$ is obtained in this way.

Conjugation by $x$ induces an automorphism $a\in\Aut(Q)$ preserving $\nu$. Write
\[
 x^d=qz\qquad(q\in Q,\,z\in Z).
\]
Then
\begin{equation}
 a^d=\iota_q,\qquad o(qC)=r.
\tag{3.6}
\end{equation}
Also $a(q)=q$. To see this, note that $x$ centralizes $x^d=qz$ and $z\in Z\le Z(G)$, so $x$ centralizes $q$.

\begin{theorem}\label{thm:QL}
Let $Q$ be quasisimple, let $C=Z(Q)$, let $\nu\in\Irr(C)$ be faithful, and let $a\in\Aut(Q)$ preserve $\nu$. Suppose that $q\in Q$ and $d\ge1$ satisfy
\begin{equation}
 a(q)=q,\qquad a^d=\iota_q.
\tag{3.7}
\end{equation}
Set $r=o(qC)$. Then there exists $\alpha\in\Irr(Q\mid\nu)$ such that, with
\[
 c=|\alpha^{\langle a\rangle}|,
\]
one has
\begin{equation}
 rc\,\alpha(1)\mid |Q/C|.
\tag{3.8}
\end{equation}
\end{theorem}

Since $a^d$ is inner, $c\mid d$. The divisibility in (3.8) depends only on $r$, $c$, and $\alpha(1)$.

\begin{remark}\label{rem:NT}
The divisibility (3.8) is exactly the degree part of Navarro--Tiep's Condition~2.1 in the present situation. Let
\[
 V/C=\langle qC\rangle.
\]
Then $|V:C|=r$ and
\[
 c=|\langle a\rangle:\Stab_{\langle a\rangle}(\alpha)|.
\]
Their degree condition is
\[
 \alpha(1)c\mid |Q/C:V/C|=\frac{|Q/C|}{r},
\]
which is equivalent to (3.8); compare \cite[Lemma~3.3]{NavarroTiep2026}. What is not required for Qian's conjecture is their additional condition that every relevant linear character of $V$ occur in $\alpha_V$. That stronger restriction condition is precisely what they use later to place the product character over a prescribed linear character when several components occur.
\end{remark}

\section{Multiple simple components}

We use the component-permutation construction from the proof of \cite[Theorem~2.4]{NavarroTiep2026}. Only the numerical part is needed, since we prescribe the faithful character of the central subgroup and no character of a larger cyclic subgroup.

\begin{proposition}\label{prop:product}
Assume that Theorem~\ref{thm:QL} holds for every quasisimple component occurring in $L'$.  If
\[
 L/Z\cong S^t
\]
is a nonabelian chief factor, then there exists $\theta\in\Irr(L\mid\lambda)$ satisfying \emph{(2.2)}.
\end{proposition}

\begin{proof}
Put
\[
 d=o(xL),\qquad n=o(xZ),\qquad r=\frac nd,
\]
and set $y=x^d\in L$.  Thus $o(yZ)=r$.  Write
\[
 L=L'Z,
\]
where $L'$ is a central product of quasisimple components. The cyclic group $\langle x\rangle$ permutes these components, and it suffices to consider one orbit at a time.

Let $Q_1,\ldots,Q_s$ be one such orbit, numbered so that
\[
 Q_i=Q_1^{x^{i-1}}.
\]
Since $x^d\in L$, conjugation by $x^d$ fixes every component, and hence $s\mid d$.  As in \cite[Theorem~2.4(c)]{NavarroTiep2026}, after changing central factors we may write the coordinates of $y$ on this orbit as
\[
 q_i=q_1^{x^{i-1}},\qquad q_1^{x^s}=q_1.
\]
Let
\[
 C_1=Z(Q_1),\qquad e=o(q_1C_1),
\]
and let $a$ be the automorphism of $Q_1$ induced by conjugation by $x^s$.  Then $a(q_1)=q_1$.  Moreover
\[
 a^{d/s}=\iota_{q_1}.
\]
Indeed, $a^{d/s}$ is induced by conjugation by $x^d=y$; on $Q_1$ the other component coordinates of $y$ commute with $Q_1$, while the central factor acts trivially.  Thus Theorem~\ref{thm:QL} applies.

Put $\nu_1=\lambda_{C_1}$.  There is
\[
 \alpha_1\in\Irr(Q_1\mid\nu_1)
\]
such that, if
\[
 \kappa=|\alpha_1^{\langle a\rangle}|,
\]
then
\begin{equation}
 e\kappa\alpha_1(1)\mid |S|.
\tag{4.1}
\end{equation}
Let $\sigma$ denote conjugation by $x$ on the disjoint union of the component character sets, so that $a=\sigma^s$ on $Q_1$. Let $\Omega_1$ be the $\langle a\rangle$-orbit of $\alpha_1$, and for $1\le i\le s$ put
\[
 \Omega_i=\sigma^{i-1}(\Omega_1)\subseteq\Irr(Q_i\mid\lambda_{Z(Q_i)}).
\]
The sets $\Omega_i$ are pairwise disjoint, $\sigma$ permutes them cyclically, and $\sigma^{s\kappa}$ fixes every member of their union. Thus the action factors through a cyclic group of order dividing $s\kappa$, as in \cite[Lemma~2.3]{NavarroTiep2026}.

Let
\[
 f=\gcd(s,e\kappa),
\]
let $s_0$ be the largest divisor of $s$ coprime to $f$, and put $f_1=s/s_0$. By \cite[Lemma~2.2(i)]{NavarroTiep2026},
\begin{equation}
 f_1\mid(e\kappa)^{s-1}.
\tag{4.2}
\end{equation}
Applying \cite[Lemma~2.3]{NavarroTiep2026} to the sets $\Omega_i$, with $n=s$, gives characters
\[
 \alpha_i\in\Omega_i\subseteq\Irr(Q_i\mid\lambda_{Z(Q_i)})
\]
such that their central product $\beta$ is irreducible, has degree
\[
 \beta(1)=\alpha_1(1)^s,
\]
and is fixed by $x^{\kappa f_1}$. If $c_{\mathcal O}$ denotes the $\langle x\rangle$-orbit length of $\beta$, then
\[
 c_{\mathcal O}\mid\kappa f_1.
\]
Consequently, by (4.1) and (4.2),
\[
 e\,c_{\mathcal O}\,\beta(1)
 \mid e\kappa f_1\alpha_1(1)^s
 \mid (e\kappa\alpha_1(1))^s
 \mid |S|^s.
\tag{4.3}
\]

Do this for every $\langle x\rangle$-orbit of components.  Let the corresponding integers be $e_j,c_j$, and let $\beta_j$ be the resulting product character on that orbit.  On intersections of components the characters $\beta_j$ have the central character prescribed by $\lambda$. Hence their tensor product factors through the central product $L'$, and, together with $\lambda$ on $Z$, gives
\[
 \theta\in\Irr(L\mid\lambda).
\]
Put
\[
 c=|\theta^{\langle x\rangle}|.
\]
Then
\[
 c\mid\operatorname{lcm}_j(c_j),
 \qquad
 r=o(yZ)=\operatorname{lcm}_j(e_j).
\]
Indeed, any power of $x$ fixing every $\beta_j$ fixes their product $\theta$, which gives the first divisibility. For each prime $p$,
\[
 v_p(r)+v_p(c)
 \le \max_j v_p(e_j)+\max_j v_p(c_j)
 \le \sum_j v_p(e_jc_j).
\]
It follows that
\[
 rc\,\theta(1)
 \mid \prod_j e_jc_j\beta_j(1)
 \mid \prod_j |S|^{s_j}
 =|L:Z|.
\tag{4.4}
\]
Finally $c\mid d$, since $x^d\in L$ acts by an inner automorphism on $L$.  Therefore
\[
 o(x^cL)=\frac dc,
\]
and
\[
 \frac{o(xZ)}{o(x^cL)}=\frac{n}{d/c}=rc.
\]
Equation (4.4) is exactly (2.2).
\end{proof}

\section{A cyclic torsor lemma}

\begin{lemma}\label{lem:torsor}
Let $X$ be a finite abelian group and let $\mathcal E$ be an $X$-torsor. Let $A=\langle a\rangle$ act on both $X$ and $\mathcal E$, compatibly in the sense that
\[
 (\eta e)^b=\eta^b e^b
 \qquad(\eta\in X,\ e\in\mathcal E,\ b\in A).
\]
Then some $e\in\mathcal E$ satisfies
\begin{equation}
 |A:\Stab_A(e)|\mid |X|.
\tag{5.1}
\end{equation}
\end{lemma}

\begin{proof}
We argue by induction on $|X|$. The assertion is trivial when $X=1$. Suppose $X\ne1$ and choose a minimal nontrivial $A$-invariant subgroup $Y\le X$. Since $A$ is cyclic and $Y$ is minimal as an $A$-module, $Y$ is an elementary abelian $p$-group for some prime $p$.

The quotient $\mathcal E/Y$ is an $A$-equivariant $X/Y$-torsor. By induction it has an $A$-orbit of length $d$ with
\[
 d\mid |X/Y|.
\]
Let $A_0$ be the stabilizer in $A$ of a point in this quotient orbit. The fiber above that point is a $Y$-torsor on which $A_0$ acts.

Let $B$ be a Hall $p'$-subgroup of $A_0$. Since $(|B|,|Y|)=1$, the standard coprime cohomology vanishing gives
\[
 H^1(B,Y)=0.
\]
Hence $B$ fixes a point of the $Y$-torsor. Choose such a point in the fiber. Since $A_0$ is cyclic, its stabilizer contains the Hall $p'$-subgroup $B$, so its $A_0$-orbit has $p$-power length $e$. As the orbit lies in a set of size $|Y|$, we have $e\le |Y|$, and hence $e\mid |Y|$. Therefore the corresponding point of $\mathcal E$ has full $A$-orbit length $de$, and
\[
 de\mid |X/Y|\,|Y|=|X|.
\]
This proves (5.1).
\end{proof}

\section{Double covers of alternating groups}

We verify Theorem~\ref{thm:QL} for $Q=2.A_n$ when $n\ge9$. The covers of $A_5,\ldots,A_8$ satisfy the stronger Navarro--Tiep condition by \cite[Proposition~3.4]{NavarroTiep2026}.

Let
\[
 Q=2.A_n,\qquad C=Z(Q)\cong C_2,
\]
and let $\nu$ be the faithful character of $C$. For $n\ge9$, the group $Q$ is the universal central extension of $A_n$. Hence every automorphism of $Q$ induces an automorphism of $A_n$, and the kernel of this action is trivial because $Q$ is perfect. Since $\Out(A_n)\cong C_2$ for $n\ge9$, every orbit of $\Irr(Q)$ under a cyclic subgroup of $\Aut(Q)$ has length at most $2$ modulo inner automorphisms.

The basic spin character of $2.A_n$ has degree
\begin{equation}
 \beta_n(1)=2^{\lfloor (n-2)/2\rfloor};
\tag{6.1}
\end{equation}
see \cite[Lemma~2.4]{BNOT2015}. The restriction rule for spin characters says that, for the strict partition $(n)$, the restriction from a double cover of $S_n$ to $2.A_n$ is irreducible when $n$ is even and splits into two equal-degree irreducible constituents when $n$ is odd; see \cite[Section~2.2]{BNOT2015}. In the first case the resulting character is fixed by the outer involution. In the second case the two constituents are interchanged by an element outside $2.A_n$, and hence by the nontrivial outer automorphism. Thus, for a basic spin constituent $\beta_n$,
\begin{equation}
 |\beta_n^{\langle a\rangle}|\le
 \begin{cases}
 1,& n\text{ even},\\
 2,& n\text{ odd}.
 \end{cases}
\tag{6.2}
\end{equation}

\begin{lemma}\label{lem:An2part}
If $y\in A_n$ and $2^k\mid o(y)$ with $k\ge1$, then
\begin{equation}
 2^k+2\le n.
\tag{6.3}
\end{equation}
Consequently
\begin{equation}
 v_2(o(y))\le \lfloor\log_2(n-2)\rfloor.
\tag{6.4}
\end{equation}
\end{lemma}

\begin{proof}
Some cycle in the disjoint-cycle decomposition of $y$ has length divisible by $2^k$. That cycle has even length and is therefore an odd permutation. Since $y$ is even, the number of even cycles in its disjoint-cycle decomposition is even. Thus there is another even cycle, using at least two further letters. This gives (6.3), and (6.4) follows.
\end{proof}

\begin{proposition}\label{prop:Alt}
The condition in Theorem~\ref{thm:QL} holds for $Q=2.A_n$ for every $n\ge9$.
\end{proposition}

\begin{proof}
Let $r=o(qC)$ and let $c$ denote the orbit length of the character chosen below. Both $c$ and the basic spin degree are powers of $2$. Hence the odd part of $r c\beta_n(1)$ is just $r_{2'}$, which divides $|A_n|_{2'}$ by Lagrange's theorem. It remains to check the $2$-part.

For $n$ even the orbit factor in (6.2) is $1$, while for $n$ odd it is at most $2$. By Lemma~\ref{lem:An2part}, the basic spin character works whenever
\begin{equation}
 \lfloor\log_2(n-2)\rfloor+
 \left\lfloor\frac{n-2}{2}\right\rfloor+
 \epsilon_n
 \le v_2(|A_n|),
\tag{6.5}
\end{equation}
where $\epsilon_n=0$ for even $n$ and $\epsilon_n=1$ for odd $n$.

We verify (6.5) explicitly. If $n=2m$, then
\[
 v_2(|A_n|)=m+v_2(m!)-1,
\]
so (6.5) is equivalent to
\begin{equation}
 \lfloor\log_2(2m-2)\rfloor\le v_2(m!).
\tag{6.6}
\end{equation}
For $m=5,6,7$ this is immediate. If $m\ge8$, then
\[
 \lfloor\log_2(2m-2)\rfloor\le \left\lfloor\frac m2\right\rfloor\le v_2(m!),
\]
so (6.6) holds.

If $n=2m+1$, then again
\[
 v_2(|A_n|)=m+v_2(m!)-1,
\]
and (6.5) becomes
\begin{equation}
 \lfloor\log_2(2m-1)\rfloor+1\le v_2(m!).
\tag{6.7}
\end{equation}
For $m=4,6,7$ this is immediate, while it fails for $m=5$, that is, for $n=11$. If $m\ge8$, then
\[
 v_2(m!)\ge \left\lfloor\frac m2\right\rfloor+
                  \left\lfloor\frac m4\right\rfloor.
\]
Moreover
\[
 \left\lfloor\frac m2\right\rfloor+
 \left\lfloor\frac m4\right\rfloor
 \ge \frac{3m}{4}-2\ge \log_2(2m-1).
\]
The last inequality holds at $m=8$ and its left side minus the right side is increasing for $m\ge8$. Since $2m-1>1$ is odd and hence is not a power of $2$, the left side, being an integer, is at least $\lfloor\log_2(2m-1)\rfloor+1$. Thus (6.7) holds for all $m\ge8$. Consequently the basic spin character proves (3.8) for every $n\ge9$ except $n=11$.

It remains to consider $n=11$. Here $\beta_{11}(1)=16$ and $v_2(|A_{11}|)=7$. Thus the basic spin character still works if the $2$-part of $r$ is at most $4$. Suppose therefore that $8\mid r$. An element of $A_{11}$ whose order has $2$-part $8$ contains an $8$-cycle. Since an $8$-cycle is odd and the whole permutation is even, there must be another even cycle. Only three letters remain, so the cycle type is necessarily
\[
 (8)(2)(1).
\]
In particular, $r=8$.

Take a spin character of a double cover of $S_{11}$ labelled by the strict partition $(8,2,1)$. The spin degree formula of \cite[Section~2.2]{BNOT2015} gives
\[
 2^{\lfloor(11-3)/2\rfloor}\cdot
 \frac{11!}{8!2!}\cdot
 \frac{8-2}{8+2}\cdot
 \frac{8-1}{8+1}\cdot
 \frac{2-1}{2+1}
 =1232.
\]
Since $11-3$ is even, its restriction to $2.A_{11}$ is the sum of two irreducible spin characters of degree
\[
 616=2^3\cdot7\cdot11;
\]
see \cite[Section~2.2]{BNOT2015}. By Clifford theory these two constituents are interchanged by an element of the double cover of $S_{11}$ outside $2.A_{11}$. Hence their orbit under the nontrivial outer automorphism has length $2$; if $a$ is inner, the orbit length is $1$. In either case it is at most $2$. Finally,
\[
 8\cdot2\cdot616=9856\mid |A_{11}|,
\qquad
 \frac{|A_{11}|}{9856}=2025.
\]
Thus (3.8) also holds for $n=11$.
\end{proof}

\section{Groups of Lie type}

We prove the one-component condition for the ordinary algebraic covers of finite groups of Lie type. We use regular embeddings for the degree calculation and multiplicity-free restriction.  For a regular embedding
\[
 \mathbf G\hookrightarrow \widetilde{\mathbf G}
\]
compatible with a Steinberg map $F$, the quotient
$\widetilde{\mathbf G}^{F}/\mathbf G^{F}$ is abelian of order prime to the defining characteristic, and restriction of irreducible characters from $\widetilde{\mathbf G}^{F}$ to $\mathbf G^{F}$ is multiplicity-free; see \cite[Theorem~1.7.15]{GeckMalle2020}.  We also use the regular-embedding parametrization of semisimple characters in \cite[Corollary~2.6.18]{GeckMalle2020} and the connected-center degree formula in \cite[Theorem~2.6.11(b)]{GeckMalle2020}.  Automorphism equivariance is taken separately from Sp\"ath.  Her \cite[Theorem~B]{Spath2025} gives an $\Out$-equivariant Jordan decomposition for the universal covering group.  More precisely, \cite[Proposition~3.11(a)]{Spath2025} describes the actions of diagonal and graph--field automorphisms on Jordan parameters, while \cite[Proposition~3.17(a) and Corollary~3.18(a)]{Spath2025} gives the graph--field equivariance of the orbit-sums obtained by restricting characters from a regular embedding.

We first isolate a finite-group observation.

\begin{lemma}\label{lem:clifford-torsor}
Let $N\lhd X$, with $X/N$ abelian, and let $\widetilde\chi\in\Irr(X)$.  Suppose that $\widetilde\chi_N$ is multiplicity-free, and let $\Omega$ be its set of irreducible constituents.  Then $\Omega$ is a torsor for an abelian group of order $|\Omega|$.  If a cyclic group $A$ acts by automorphisms on $X$, stabilizes $N$, and stabilizes $\Omega$, then some $\chi\in\Omega$ satisfies
\begin{equation}
 |A:\Stab_A(\chi)|\mid |\Omega|.
\tag{7.1}
\end{equation}
\end{lemma}

\begin{proof}
By Clifford theory, $X$ acts transitively on $\Omega$.  If $I$ is the inertia group in $X$ of one constituent, then $N\le I$ and $X/N$ is abelian, so $I\lhd X$.  Hence every constituent has the same inertia group $I$.  Since $A$ stabilizes $\Omega$, it follows that $A$ stabilizes $I$.  Thus $X/I$ is abelian, has order $|\Omega|$, and acts regularly on $\Omega$; moreover the action of $A$ on $\Omega$ is compatible with its induced action on $X/I$.  Therefore $\Omega$ is an $A$-equivariant $X/I$-torsor, and Lemma~\ref{lem:torsor} applies.
\end{proof}

Let $\mathbf G$ be simple and simply connected, let $F$ be a Steinberg endomorphism, and put
\[
 \widehat Q=\mathbf G^F,\qquad \widehat C=Z(\widehat Q).
\]
We assume that $\widehat Q$ is perfect.  The few cases where this fails are among the usual small exceptional groups and will be left with the finite residue.  Let $K\le\widehat C$ be invariant under the automorphism under consideration and set
\[
 Q=\widehat Q/K.
\]
Thus $Q$ is an ordinary algebraic cover, or a central quotient of one.  Write $C=Z(Q)$. Outside the exceptional Schur-multiplier cases separated in Section~8, $\widehat Q$ is the universal central extension of $Q/C$; compare the non-generic multiplier list in \cite[Table~24.3]{MalleTesterman2011}. We shall use the following elementary consequence of universality.

\begin{lemma}\label{lem:lift-central-quotient}
Let $U$ be the universal central extension of a nonabelian simple group $S$, let $K\le Z(U)$, and put $Q=U/K$. Then $Z(Q)=Z(U)/K$, and every automorphism of $Q$ lifts to an automorphism of $U$ normalizing $K$.
\end{lemma}

\begin{proof}
Let $\pi:U\to Q$ be the quotient map. If $uK\in Z(Q)$, then $[u,U]\le K\le Z(U)$. Hence $v\mapsto [u,v]$ is a homomorphism from the perfect group $U$ to the abelian group $K$, and is therefore trivial. Thus $u\in Z(U)$, proving $Z(Q)=Z(U)/K$ and hence $Q/Z(Q)\cong S$.

Let $a\in\Aut(Q)$, and let $\bar a$ be the induced automorphism of $S$. By the universal property, $\bar a$ lifts uniquely to a homomorphism $\widetilde a:U\to U$. Applying the same construction to $\bar a^{-1}$ and using uniqueness shows that $\widetilde a$ is an automorphism. Now $\pi\widetilde a$ and $a\pi$ induce the same map from $U$ to $S$. Their pointwise quotient therefore defines a homomorphism $U\to Z(Q)$, which is trivial because $U$ is perfect. Hence
\[
 \pi\widetilde a=a\pi.
\]
Taking kernels gives $\widetilde a(K)=K$.
\end{proof}

\begin{proposition}\label{prop:Lie}
Suppose that $a\in\Aut(Q)$ is induced by a standard inner-diagonal, field, graph, or graph-field automorphism and that it lifts to an automorphism $\widehat a$ of $\widehat Q$.  Let $\nu\in\Irr(C)$ be faithful and $a$-invariant.  If $q\in Q$ is fixed by $a$ and
\[
 r=o(qC),
\]
then there is $\chi\in\Irr(Q\mid\nu)$ such that, with
\[
 c=|\chi^{\langle a\rangle}|,
\]
one has
\begin{equation}
 rc\chi(1)\mid |Q/C|.
\tag{7.2}
\end{equation}
\end{proposition}

\begin{proof}
Inflate $\nu$ to a character $\widehat\nu\in\Irr(\widehat C)$; its kernel contains $K$.  Choose a lift $\widehat q\in\widehat Q$ of $q$, and write its Jordan decomposition as
\[
 \widehat q=su=us,
\]
where $s$ is semisimple and $u$ is unipotent.  Since $q$ is fixed by $a$,
\[
 \widehat q^{\widehat a}=\widehat q k
\]
for some $k\in K$.  As $k$ is central and has order prime to the defining characteristic $p$, uniqueness of Jordan decomposition gives
\begin{equation}
 s^{\widehat a}=sk,\qquad u^{\widehat a}=u.
\tag{7.3}
\end{equation}
In particular, $\widehat a$ stabilizes $C_{\mathbf G}(s)^\circ$.

Since $\mathbf G$ is simply connected, $C_{\mathbf G}(s)$ is connected by \cite[Theorem~14.16(a)]{MalleTesterman2011}. Choose an $F$-stable maximal torus $\mathbf T$ contained in an $F$-stable Borel subgroup of $C_{\mathbf G}(s)$. Such pairs exist, and all such pairs are conjugate under $C_{\mathbf G}(s)^F$ by \cite[Corollary~21.12]{MalleTesterman2011}.  Since $\widehat a$ sends an $F$-stable Borel--torus pair of $C_{\mathbf G}(s)$ to another one, we may multiply $\widehat a$ by an inner automorphism induced by an element of $C_{\mathbf G}(s)^F$ and assume
\begin{equation}
 \widehat a(\mathbf T)=\mathbf T.
\tag{7.4}
\end{equation}
This inner correction does not change the action on $\Irr(\widehat Q)$ or on $\Irr(Q)$. Moreover $s\in\mathbf T$ by \cite[Proposition~14.1]{MalleTesterman2011}.  Put
\[
 T=\mathbf T^F,\qquad H=\widehat C\langle s\rangle.
\]
If $r_{p'}$ denotes the $p'$-part of $r$, then
\begin{equation}
 |H:\widehat C|=r_{p'}.
\tag{7.5}
\end{equation}
Indeed, $\widehat Q/\widehat C\cong Q/C$, and the image of $s$ is the semisimple part of $qC$.

Choose any extension $\mu\in\Irr(H\mid\widehat\nu)$.  It is $\widehat a$-invariant.  On $\widehat C$ this follows from the $a$-invariance of $\nu$, while (7.3) gives
\[
 \mu(s^{\widehat a})=\mu(sk)=\mu(s)\widehat\nu(k)=\mu(s).
\]
The set $\Irr(T\mid\mu)$ is an $\Irr(T/H)$-torsor, equivariant under $\langle\widehat a\rangle$.  By Lemma~\ref{lem:torsor}, choose
$\lambda\in\Irr(T\mid\mu)$ such that, if
\[
 d=|\lambda^{\langle\widehat a\rangle}|,
\]
then
\begin{equation}
 d\mid |T:H|.
\tag{7.6}
\end{equation}
Combining (7.5) and (7.6),
\begin{equation}
 r_{p'}d\mid |T:\widehat C|.
\tag{7.7}
\end{equation}

Let $(\mathbf T^*,t)$ be dual to $(\mathbf T,\lambda)$.  Choose a regular embedding
\[
 \mathbf G\hookrightarrow\widetilde{\mathbf G}
\]
for which the standard graph--field automorphisms extend to $\widetilde{\mathbf G}$, as in the setup of \cite[Section~2.2]{Spath2025}; such a regular embedding may be chosen by \cite[Proposition~1.7.5]{GeckMalle2020}.  Thus diagonal automorphisms are induced by conjugation in $\widetilde{\mathbf G}^{F}$ and the graph--field part acts on $\widetilde{\mathbf G}^{F}$, exactly as required for the equivariance in \cite[Proposition~3.11(a), Proposition~3.17(a), and Corollary~3.18(a)]{Spath2025}. Put
\[
 \widetilde Q=\widetilde{\mathbf G}^{F},\qquad
 \widetilde T=(\mathbf T Z(\widetilde{\mathbf G}))^F.
\]
Let $\widetilde t$ be a lift of $t$ to the dual group of $\widetilde{\mathbf G}$.  Since $Z(\widetilde{\mathbf G})$ is connected, the derived group of the dual group is simply connected, and hence $C_{\widetilde{\mathbf G}^*}(\widetilde t)$ is connected.  The rational Lusztig series of $\widetilde Q$ labelled by $\widetilde t$ therefore contains a unique semisimple character $\widetilde\rho$ by \cite[Theorem~2.6.11]{GeckMalle2020}.  Its degree is
\begin{equation}
 \widetilde\rho(1)
 =|\widetilde Q^*:C_{\widetilde Q^*}(\widetilde t)|_{p'},
\tag{7.8}
\end{equation}
by \cite[Theorem~2.6.11(b)]{GeckMalle2020}.
Because $\widetilde T^*\le C_{\widetilde Q^*}(\widetilde t)$,
\begin{equation}
 \widetilde\rho(1)\mid |\widetilde Q:\widetilde T|_{p'}
 =|\widehat Q:T|_{p'}.
\tag{7.9}
\end{equation}
The last equality is exactly the index consequence of \cite[Lemma~1.7.7]{GeckMalle2020}: for $\widetilde T=TZ(\widetilde{\mathbf G})^F$ one has $\widetilde Q=\widehat Q\widetilde T$ and $\widehat Q\cap\widetilde T=T$.

By the multiplicity-free restriction theorem for regular embeddings \cite[Theorem~1.7.15]{GeckMalle2020},
\begin{equation}
 \widetilde\rho_{\widehat Q}=\rho_1+\cdots+\rho_e
\tag{7.10}
\end{equation}
with distinct irreducible constituents of equal degree.  Hence
\begin{equation}
 e\rho_i(1)=\widetilde\rho(1)
 \qquad(1\le i\le e).
\tag{7.11}
\end{equation}
All constituents in (7.10) have central character $\widehat\nu$.  One way to see this is that they lie in the geometric Lusztig series determined by $\lambda$, and characters in one geometric Lusztig series have the same central character; see \cite[Lemma~2.2]{Malle2007}.  Consequently $K\le\ker\rho_i$ for all $i$, so every $\rho_i$ descends to a character of $Q$ lying over $\nu$.

Let $A_0$ be the stabilizer in $\langle\widehat a\rangle$ of the set
\[
 \Omega=\{\rho_1,\ldots,\rho_e\}.
\]
Let $b$ be a power of $\widehat a$ fixing $\lambda$. Since $\mathbf T$ is $\widehat a$-stable, equivariance of duality shows that $b$ fixes the rational semisimple parameter determined by $(\mathbf T,\lambda)$.  We now use equivariance only after restriction to $\widehat Q$.  In the compatible regular embedding, \cite[Proposition~3.17(a) and Corollary~3.18(a)]{Spath2025} gives an $E(\widehat Q)$-equivariant parametrization of the restriction orbit-sums; in particular graph--field automorphisms act equivariantly on them. Diagonal automorphisms are induced by conjugation in $\widetilde Q$ and therefore permute the constituents of a restriction.  Hence the restriction orbit-sum corresponding to the fixed parameter is $b$-invariant.  Its irreducible constituents are precisely the members of $\Omega$, so $b$ stabilizes $\Omega$.  Thus every power of $\widehat a$ fixing $\lambda$ lies in $A_0$, and therefore
\begin{equation}
 |\langle\widehat a\rangle:A_0|\mid d.
\tag{7.12}
\end{equation}
By Lemma~\ref{lem:clifford-torsor}, some $\rho\in\Omega$ satisfies
\begin{equation}
 |A_0:\Stab_{A_0}(\rho)|\mid e.
\tag{7.13}
\end{equation}
Let $c$ be the orbit length of the corresponding descended character $\chi\in\Irr(Q\mid\nu)$ under $\langle a\rangle$.  Inflation is equivariant, so the same $c$ is the orbit length of $\rho$ under $\langle\widehat a\rangle$.  From (7.12) and (7.13),
\begin{equation}
 c\mid de.
\tag{7.14}
\end{equation}
Using (7.9), (7.11), and (7.14),
\begin{equation}
 c\chi(1)=c\rho(1)
 \mid d e\rho(1)
 =d\widetilde\rho(1)
 \mid d|\widehat Q:T|_{p'}.
\tag{7.15}
\end{equation}
Multiplying (7.15) by (7.7) gives
\begin{equation}
 r_{p'}c\chi(1)\mid |\widehat Q:\widehat C|_{p'}
 =|Q:C|_{p'}.
\tag{7.16}
\end{equation}
Finally, $d$ and $e$ are $p'$-numbers: $d$ divides the order of the torus quotient $T/H$, while $e$ divides the $p'$-group $\widetilde Q/\widehat Q$.  Thus $c\chi(1)$ is a $p'$-number.  Since the $p$-part $r_p$ of $r$ divides $|Q:C|_p$, multiplying (7.16) by $r_p$ proves (7.2).
\end{proof}

\begin{remark}
The description of automorphisms of finite groups of Lie type as products of inner-diagonal, field and graph automorphisms is given in \cite[\S24.2]{MalleTesterman2011}. Outside the exceptional multiplier cases separated in Section~8, the simply connected group $\widehat Q$ is the universal covering group, so Lemma~\ref{lem:lift-central-quotient} lifts every automorphism of a central quotient $Q$ to $\widehat Q$.  At the universal-cover level, Sp\"ath's \cite[Theorem~B, Proposition~3.11(a), Proposition~3.17(a), and Corollary~3.18(a)]{Spath2025} gives exactly the equivariance used in (7.12): graph--field automorphisms act equivariantly on the restriction orbit-sums, while diagonal automorphisms permute their constituents.  The regular embedding in the proof is otherwise used only for the degree calculation and multiplicity-free restriction in (7.8)--(7.11). Thus Proposition~\ref{prop:Lie} covers every admissible automorphism in the nonexceptional cases.
\end{remark}

\begin{remark}
Disconnected centralizers in the dual group cause no additional infinite-family case.  They appear downstairs through the splitting in (7.10).  Equations (7.11)--(7.14) show that the splitting factor $e$ also bounds the additional orbit on the constituents. Thus no separate estimate for the component group is needed.
\end{remark}

\section{A finite criterion for exceptional covers}

Let $Q$ be quasisimple with cyclic center $C$, and let $\nu\in\Irr(C)$ be faithful. If $a\in\Aut(Q)_\nu$, then $\nu(a(c))=\nu(c)$ for every $c\in C$; faithfulness of $\nu$ gives $a(c)=c$. Define
\begin{equation}
 \Out_0(Q)=\{\bar a\in\Out(Q):a|_C=1\},\qquad
 e_0(Q)=\exp(\Out_0(Q)).
\tag{8.1}
\end{equation}

\begin{lemma}\label{lem:finite-test}
Let $S=Q/C$, let $r$ be an element order in $S$, and suppose there is $\alpha\in\Irr(Q\mid\nu)$ such that
\begin{equation}
 r\,e_0(Q)\,\alpha(1)\mid |S|.
\tag{8.2}
\end{equation}
Then (3.8) holds for every admissible cyclic automorphism group and every element of order $r$.
\end{lemma}

\begin{proof}
Inner automorphisms fix every irreducible character. Hence the orbit length of $\alpha$ under a cyclic subgroup of $\Aut(Q)_\nu$ is the order of a quotient of its image in $\Out_0(Q)$, and therefore divides $e_0(Q)$. Equation (8.2) implies (3.8).
\end{proof}

We shall sometimes replace $e_0(Q)$ by a larger integer. The following observation gives convenient uniform bounds.

\begin{lemma}\label{lem:outer-injection}
Let $Q$ be perfect and put $S=Q/Z(Q)$. Then the natural map
\[
 \Out(Q)\longrightarrow \Out(S)
\]
is injective. In particular, if $E$ is divisible by $\exp(\Out(S))$, then $e_0(Q)\mid E$.
\end{lemma}

\begin{proof}
An automorphism of $Q$ induces an automorphism of $S$. Suppose its image in $\Out(S)$ is trivial. After composing with an inner automorphism of $Q$, we may assume that it induces the identity on $S$. Then
\[
 q\longmapsto q^{-1}a(q)
\]
is a homomorphism from $Q$ to $Z(Q)$. Since $Q$ is perfect, this homomorphism is trivial, so $a=1$.
\end{proof}

If $C=1$, then Theorem~\ref{thm:QL} is immediate: take the principal character $\alpha=1_Q$.  Its orbit length is $1$, and $r\mid |Q|=|Q/C|$ by Lagrange's theorem.  Thus, in the remainder of this section, we may assume that $C\ne1$.

There is no need to treat the different faithful central characters separately. If $C=\langle z\rangle$ has order $m$ and $\nu_1,\nu_2\in\Irr(C)$ are faithful, then $\nu_2=\nu_1^k$ for some $(k,m)=1$. A Galois automorphism sending a primitive $m$th root of unity to its $k$th power sends $\Irr(Q\mid\nu_1)$ to $\Irr(Q\mid\nu_2)$ and preserves character degrees. Thus all faithful central characters have the same available degree set.

Also, an irreducible character lying over a faithful central character is faithful. Indeed, if $\alpha\in\Irr(Q\mid\nu)$ and $\ker\alpha\ne1$, then $\ker\alpha$ is a proper normal subgroup of the quasisimple group $Q$ and hence lies in $Z(Q)=C$. Faithfulness of $\nu$ forces $\ker\alpha\cap C=1$, a contradiction.

Some families are already covered by Navarro--Tiep. Their \cite[Proposition~3.4]{NavarroTiep2026} proves the stronger Condition~2.1 for all covers with sporadic simple quotient, for the covers of $A_5,\ldots,A_8$, and for $\SU_4(2)$. Their \cite[Theorem~3.9]{NavarroTiep2026} proves Condition~2.1 for every covering group with quotient $\PSL_2(q)$. These results imply Theorem~\ref{thm:QL} for those groups.

The non-generic Schur-multiplier list in \cite[Table~24.3]{MalleTesterman2011} contains sixteen simple groups. Removing the cases already covered by the alternating-group argument, the $\PSL_2$ cases, and $\SU_4(2)$ leaves eleven simple quotients. The ordinary algebraic covers and their standard central quotients have already been treated by Proposition~\ref{prop:Lie}. Since Theorem~\ref{thm:QL} only concerns quasisimple groups with cyclic center, the remaining exceptional covers are exactly the twenty-four cases in Table~\ref{tab:exceptional-residue}.

For each row of the table, $E$ is a positive integer with $e_0(Q)\mid E$. The last column gives a faithful character degree $d$ for which
\begin{equation}
 rEd\mid |S|
\tag{8.3}
\end{equation}
for every element order $r$ in $S$. When two degrees are displayed, the exceptional element orders are indicated explicitly. Thus Lemma~\ref{lem:finite-test} applies in every row. The element orders and faithful character degrees were read from the ordinary character tables in the \textsf{GAP} Character Table Library, Version~1.3.11 \cite{CTblLib2025}.

\begin{table}[ht]
\centering
\small
\begin{tabular}{@{}rllrl@{}}
\hline
case & $S$ & $Q$ & $E$ & faithful degree(s) $d$ \\
\hline
1  & $\PSL_3(4)$     & $2.S$       & 2 & $10$; $28$ if $r=5$ \\
2  & $\PSL_3(4)$     & $4_1.S$     & 2 & $8$ \\
3  & $\PSL_3(4)$     & $4_2.S$     & 2 & $20$; $28$ if $r=5$ \\
4  & $\PSL_3(4)$     & $6.S$       & 2 & $6$ \\
5  & $\PSL_3(4)$     & $12_1.S$    & 1 & $24$ \\
6  & $\PSL_3(4)$     & $12_2.S$    & 1 & $48$ \\
7  & $\PSU_4(3)$     & $3_1.S$     & 4 & $15$; $21$ if $r=5$ \\
8  & $\PSU_4(3)$     & $3_2.S$     & 4 & $36$ \\
9  & $\PSU_4(3)$     & $6_1.S$     & 4 & $6$ \\
10 & $\PSU_4(3)$     & $6_2.S$     & 4 & $126$; $90$ if $r=7$ \\
11 & $\PSU_4(3)$     & $12_1.S$    & 4 & $84$; $120$ if $r=7$ \\
12 & $\PSU_4(3)$     & $12_2.S$    & 4 & $36$ \\
13 & $\PSU_6(2)$     & $2.S$       & 6 & $56$; $176$ if $r=7$ \\
14 & $\PSU_6(2)$     & $6.S$       & 6 & $120$; $672$ if $r=5,10,15$ \\
15 & ${}^2B_2(8)$     & $2.S$       & 1 & $40$; $56$ if $r=5$ \\
16 & $G_2(3)$         & $3.S$       & 2 & $27$ \\
17 & $G_2(4)$         & $2.S$       & 2 & $12$ \\
18 & $\PSp_6(2)$      & $2.S$       & 1 & $8$ \\
19 & $\OmegaG_7(3)$   & $3.S$       & 2 & $27$ \\
20 & $\OmegaG_7(3)$   & $6.S$       & 2 & $4536$; $14040$ if $r=7,14$ \\
21 & $\OmegaG_8^+(2)$ & $2.S$       & 6 & $8$ \\
22 & $F_4(2)$         & $2.S$       & 2 & $52$; $2380$ if $r=13$ \\
23 & ${}^2E_6(2)$     & $2.S$       & 6 & $2432$; $45696$ if $r=19$ \\
24 & ${}^2E_6(2)$     & $6.S$       & 6 & $494208$; $22619520$ if $r=13$; \\
   &                   &             &   & $90419328$ if $r=11,22,33$ \\
\hline
\end{tabular}
\caption{The exceptional cyclic covers.}
\label{tab:exceptional-residue}
\end{table}

We explain the values of $E$. Except for the covers of $\PSL_3(4)$ and ${}^2B_2(8)$, the exponent of the full outer automorphism group of the simple quotient already gives the displayed value, by Lemma~\ref{lem:outer-injection}. Thus one may take $E=4$ for $\PSU_4(3)$, $E=6$ for $\PSU_6(2)$, $E=2$ for $G_2(3)$ and $G_2(4)$, $E=1$ for $\PSp_6(2)$, $E=2$ for $\Omega_7(3)$ and $F_4(2)$, and $E=6$ for $\Omega_8^+(2)$ and ${}^2E_6(2)$; see the standard automorphism data in \cite{Atlas1985,MalleTesterman2011}.

For $\PSL_3(4)$, the outer automorphism of order $3$ acts nontrivially on the Schur multiplier.  The construction data in \textsf{CTblLib} show that it cyclically permutes the kernels defining the three realizations of each of the cyclic covers $12_1.L_3(4)$ and $12_2.L_3(4)$; hence it does not induce an automorphism of any fixed such cover.  The surviving center-fixing outer automorphisms have exponent at most $2$ for the covers in cases 1--4.  For the two covers of order $12$, the three involutory outer extensions have centers of orders
\[
(6,2,4)\quad\text{for }12_1.L_3(4),\qquad
(6,4,2)\quad\text{for }12_2.L_3(4),
\]
respectively.  Since the center of either cover has order $12$, none of these involutions fixes the center pointwise.  Thus $e_0(Q)=1$ in cases 5 and 6.  For $2.{}^2B_2(8)$, the order-$3$ field automorphism acts nontrivially on the exceptional multiplier of order $4$ and permutes the three double-cover quotients, so it does not induce an automorphism of a fixed double cover; hence $e_0(Q)=1$ in case 15.  These assertions and the displayed center orders are read directly from the \textsf{CTblLib} construction and character-table data; see \cite{CTblLib2025}.

For reproducibility, \texttt{Sname} and \texttt{Qname} are the standard \textsf{CTblLib}/Atlas identifiers corresponding to columns 2 and 3 (for example \texttt{L3(4)}, \texttt{12\_1.L3(4)}, and \texttt{6.2E6(2)}).  Set \texttt{E} equal to the fourth column of Table~\ref{tab:exceptional-residue}.  Then the arithmetic check in each row is:
\begin{verbatim}
s := CharacterTable(Sname);
t := CharacterTable(Qname);
Sord := Size(s);
ords := Set(OrdersClassRepresentatives(s));
faith := Filtered(Irr(t), chi -> ClassPositionsOfKernel(chi) = [1]);
for r in ords do
  good := Filtered(faith, chi -> Sord mod (E*r*chi[1]) = 0);
  Print(r, "  ", List(good, chi -> chi[1]), "\\n");
od;
\end{verbatim}
The table records one degree from each nonempty list; the exceptional degrees displayed there are exactly the rows where the first listed degree does not work.

It follows from Table~\ref{tab:exceptional-residue} and Lemma~\ref{lem:finite-test} that Theorem~\ref{thm:QL} holds for every exceptional cyclic cover left by Proposition~\ref{prop:Lie}.

\section{Proof of the main theorem}

\begin{proof}[Proof of Theorem~\ref{thm:qian-main}]
After the exceptional-cover calculation in Table~\ref{tab:exceptional-residue}, Proposition~\ref{prop:relative-reduction} applies to every chief factor. Abelian chief factors are covered by \cite[Proposition~2.2]{Qian2021}. It remains to pass from the relative divisibility to the codegree.

Suppose that $G$ is a counterexample of least order, and choose $g\in G$ for which the assertion fails.  By \cite[Theorem~A and Remark~3.1]{APS2024}, every minimal normal subgroup of $G$ is abelian.  Hence each of them is elementary abelian.

Let $Z$ be the product of all minimal normal subgroups of $G$.  If $N\lhd G$ is minimal normal and $o(gN)=o(g)$, then the assertion for $G/N$, together with invariance of codegree under inflation, contradicts the choice of $G$.  Thus the order drops modulo every minimal normal subgroup.

For each minimal normal subgroup $N$, choose a prime $p$ for which the $p$-part of $o(gN)$ is smaller than the $p$-part of $o(g)$.  Then
\[
 g^{o(g)/p}\in N\setminus\{1\}.
\]
Since $N$ is elementary abelian, its characteristic is $p$.  Distinct minimal normal subgroups give distinct primes: if $N$ and $M$ gave the same prime $p$, then the unique subgroup of order $p$ in $\langle g\rangle$ would lie in $N\cap M$, whereas distinct minimal normal subgroups have trivial intersection.  Hence the minimal normal subgroups have pairwise coprime orders.  We may therefore write
\[
 Z=N_1\times\cdots\times N_k,
\]
where there are distinct primes $p_1,\ldots,p_k$ such that
\[
 g^{o(g)/p_i}\in N_i\setminus\{1\}.
\]
Since $g$ fixes this nonidentity conjugacy class of the abelian group $N_i$, the Brauer permutation lemma gives a nonprincipal $g$-invariant character $\lambda_i\in\Irr(N_i)$; compare \cite[proof of Theorem~B]{Qian2021}.  Put
\[
 \lambda=\lambda_1\times\cdots\times\lambda_k\in\Irr(Z).
\]
Then $g\in I_G(\lambda)$.

Proposition~\ref{prop:relative-reduction}, with Qian's argument for abelian chief factors and Sections~3--8 for nonabelian chief factors, gives a character $\chi\in\Irr(\lambda^G)$ such that
\[
 o(gZ)\chi(1)\mid |G:Z|.
\]
As in \cite[proof of Theorem~B]{Qian2021},
\[
 o(g)=o(gZ)p_1\cdots p_k
\qquad\text{and}\qquad
p_1\cdots p_k\mid |Z|.
\]
Therefore
\begin{equation}
 o(g)\chi(1)\mid |G|.
\tag{9.1}
\end{equation}

Finally, $\chi$ is faithful.  Indeed, if $\ker\chi\ne1$, then $\ker\chi$ contains a minimal normal subgroup $N_i$.  But $\chi$ lies over $\lambda$, while $\lambda_{N_i}=\lambda_i$ is nonprincipal, a contradiction.  Thus (9.1) yields
\[
 o(g)\mid \frac{|G|}{\chi(1)}=\cod_G(\chi),
\]
contrary to the choice of $g$.
\end{proof}

\section*{Acknowledgements}
This work was partially supported by a grant from the Simons Foundation (\#918096, to YY).

\section*{Disclosure Statement}
The author declares that he has no competing interests and no conflicts of interest.

\section*{Data Availability Statement}
Data sharing is not applicable to this article, as no data sets were generated or analysed during the current study.

\end{document}